\documentclass[11pt]{amsart}
\usepackage{amsmath}
\usepackage{}
\usepackage{amsfonts}
\usepackage{amsfonts,latexsym,rawfonts,amsmath,amssymb,amsthm}
\usepackage[plainpages=false]{hyperref}

\usepackage{graphicx}
\makeatletter
\renewcommand\@biblabel[1]{}
\makeatother

\numberwithin{equation}{section}
\newcommand{\beq}{\begin{equation}}
\newcommand{\eeq}{\end{equation}}
\newcommand{\beqs}{\begin{eqnarray*}}
\newcommand{\eeqs}{\end{eqnarray*}}
\newcommand{\beqn}{\begin{eqnarray}}
\newcommand{\eeqn}{\end{eqnarray}}
\newcommand{\beqa}{\begin{array}}
\newcommand{\eeqa}{\end{array}}

\def\lra{\longrightarrow}

\def\bc{\begin{center}}
\def\ec{\end{center}}

\def\begeq{\begin{equation}}
\def\endeq{\end{equation}}
\def\and{\quad{\rm and}\quad}

\let\lra=\longrightarrow

\def\mapright\#1{\,\smash{\mathop{\lra}\limits^{\#1}}\,}

\newtheorem{prop}{Proposition}[section]
\newtheorem{theo}[prop]{Theorem}
\newtheorem{lem}[prop]{Lemma}
\newtheorem{claim}[prop]{Claim}
\newtheorem{cor}[prop]{Corollary}
\newtheorem{rem}[prop]{Remark}

\newtheorem{defi}[prop]{Definition}

\newtheorem{problem}[prop]{Problem}

\begin{document}

\date{}
\author {Yuxing Deng }
\author { Xiaohua $\text{Zhu}^*$}

\thanks {* Partially supported by the NSFC Grants 11271022 and 11331001}

\address{ Yuxing Deng\\School of Mathematical Sciences, Beijing Normal University,
Beijing, 100875, China\\
dengyuxing@mail.bnu.edu.cn}

\address{ Xiaohua Zhu\\School of Mathematical Sciences and BICMR, Peking University,
Beijing, 100871, China\\
xhzhu@math.pku.edu.cn}

\title{ Asymptotic behavior of positively curved steady Ricci Solitons}
\maketitle

\section*{\ }

\begin{abstract} In this paper, we analyze the asymptotic behavior of $\kappa$-noncollapsed and positively curved steady Ricci solitons and prove that any $n$-dimensional $\kappa$-noncollapsed steady K\"ahler-Ricci soliton with non-negative sectional curvature must be flat.
\end{abstract}

\section{Introduction}

The classification of positively curved steady soliton is an important problem in the study of Ricci flow. In his celebrated paper \cite{Pe1}, Perelman conjectured that \textit{all 3-dimensional $\kappa$-noncollapsed steady (gradient) Ricci solitons must be rotationally symmetric} (Precisely, Perelman claims that the conjecture is true without giving any sketch of proof, see \textbf{11.9} of that paper).
The conjecture was solved by Brendle in 2012 \cite{Br1}. Brendle also proved that the same result holds for higher dimensional $\kappa$-noncollapsed Ricci solitons
with nonnegative sectional curvature if they are asymptotically cylindrical \cite{Br2}. Under the condition of locally conformally flat condition, Cao and Chen also proved the rotational symmetry of gradient steady soltion  \cite{CaCh}. These rotationally symmetric metrics are usually called the Bryant steady Ricci solitons.

More than 20 years ago,  Cao  constructed a family of $U(n)$-invariant steady K\"ahler-Ricci solitons with positive sectional curvature on $\mathbb{C}^{n}$ \cite{Ca2}. He also proposed the following open problem:
\begin{problem}\label{question-1}
Is it true that any complete gradient steady K\"{a}hler-Ricci soliton with positive sectional curvature must be $U(n)$-invariant?
\end{problem}
Unlike the Bryant solitons, one can check that Cao's solitons are all collapsed (cf. Appendix ). Thus, it is interesting to ask the following question:

\begin{problem}\label{question-2}
Does there exist a steady K\"{a}hler-Ricci soliton with positive sectional curvature which is $\kappa$-noncollapsed?
\end{problem}
In this paper, we give a negative answer to Problem \ref{question-2}. Namely we prove

\begin{theo}\label{main-theorem-nonexistence-1}There is no
$\kappa$-noncollapsed steady gradient K\"{a}hler-Ricci soliton with positive sectional curvature.
\end{theo}

Theorem \ref{main-theorem-nonexistence-1} gives a positive evidence to Problem \ref{question-1}. As an application of Theorem \ref{main-theorem-nonexistence-1}, we get the following rigidity result.

\begin{theo}\label{theorem-eternal-flow}
Any $\kappa$-noncollapsed steady K\"ahler-Ricci soliton with nonnegative sectional curvature must be flat. More generally, any $\kappa$-noncollapsed noncompact and eternal K\"{a}hler-Ricci flow with nonnegative sectional curvature and uniformly bounded curvature must be a flat flow.
\end{theo}

We use the induction argument to prove Theorem \ref{main-theorem-nonexistence-1} and first prove it for K\"ahler surfaces. The main technique is to analyze asymptotic behavior of positively curved steady Ricci solitons as used by many people, such as in \cite {Pe1}, \cite{N}, \cite {Na},  \cite{Br1},  etc.. By the blow-down argument, we first generalize Perelman's compactness theorem  for  3-dimensional $\kappa$-solution in \cite{Pe1} to higher dimensions (see Theorem \ref{compactness-theorem}). Then we apply the compactness theorem to the steady solitons and prove

\begin{theo}\label{soliton-1}
Let $(M,g, f)$ be a noncompact $\kappa$-noncollapsed steady K\"{a}hler-Ricci soliton with dimension $n$. Suppose that $M$ has nonnegative sectional curvature and positive Ricci curvature.
Then, for any $p_{i}\rightarrow\infty$, the sequence of rescaled flows $(M,R(p_{i})g(R^{-1}(p_{i})t);  p_{i})$ converges subsequently to a K\"{a}hler-Ricci flow
$(N_{1}\times N_{2},\widetilde{g}(t))$ ( $t\in (-\infty,0]$) in the Cheeger-Gromov topology, where
$$\widetilde{g}(t)={\rm d}z\otimes{\rm d}\overline{z}+g_{N_2}(t),$$
$N_{1}$
is $\mathbb C^1$ or $\mathbb R^1\times S^1$ with
the flat metric $g_{N_1}={\rm d}z\otimes{\rm d}\overline{z}$,  and $(N_{2},g_{N_{2}}(t))$ is a pseudo $\kappa$-solution (cf. Definition \ref{def-pesudo-solution}) of K\"{a}hler-Ricci flow on a complex manifold $N_2$ with dimension $n-1$.
Furthermore, in case ${\rm dim}_{\mathbb C} M=2$, $(N_2,g_{N_2}(t))=(\mathbb{C}\mathbb{P}^{1},(1-t)g_{FS})$, where $g_{FS}$ is the Fubini-Study metric of $\mathbb CP^1$.
\end{theo}

Once Theorem \ref{soliton-1} is available, we study integral curves generated by the Killing vector field $J\nabla f$ on $(M,g,f)$. We show that there exists
a sequence of closed integral curves whose lengths has a positive lower bound under suitable rescaled metrics of $g$. On the other hand, we can use the global Poincar\'{e} coordinates on $M$ constructed by Bryant in
\cite {Bry} to prove that the length of those curves should tend to zero. This will lead to a contradiction!

We remark that the real version of Theorem \ref{soliton-1} is also true.

\begin{theo}\label{theorem-soliton-real}
Let $(M,g, f)$ be a noncompact $\kappa$-noncollapsed steady Ricci soliton with dimension $n$. Suppose that $M$ has nonnegative curvature operator and positive Ricci curvature. We also assume that $(M,g,f)$ has a unique equilibrium point.
Then, for any $p_{i}\rightarrow\infty$, the sequence of rescaled flows $(M,R(p_{i})g(R^{-1}(p_{i})t);  p_{i})$ converges subsequently to a Ricci flow
$(\mathbb{R}\times N,\widetilde{g}(t))$ ( $t\in (-\infty,0]$) in the Cheeger-Gromov topology, where
$$\widetilde{g}(t)={\rm d}s\otimes{\rm d}s+g_{N}(t),$$
and $(N,g_{N}(t))$ is a pseudo $\kappa$-solution on $N$ with dimension $n-1$.
\end{theo}

The proof of Theorem \ref{theorem-soliton-real} is the same as Theorem \ref{soliton-1}. Theorem \ref{theorem-soliton-real} gives an asymptotic behavior of $\kappa$-noncollapsed steady solitons with nonnegative curvature operator in higher dimensions.

The paper is organized as follows. In Section 2, we recall some facts on $\kappa$-solution. In Section 3, we give a generalization of Perelman's compactness theorem to  higher dimensional $\kappa$-solutions. In Section 4, we analyze the asymptotic geometry of steady solitons and prove Theorem \ref{soliton-1}. Both of Theorem \ref{main-theorem-nonexistence-1} and Theorem \ref{theorem-eternal-flow} are  proved in Section 5.

In a subsequel of papers \cite{DZ2}, we improve Theorem \ref{main-theorem-nonexistence-1} as follows.

\begin{theo}\label{main-theorem-improved}There is no
$\kappa$-noncollapsed steady gradient K\"{a}hler-Ricci soliton with positive  bisectional curvature.
\end{theo}
\vskip6mm

\section{Preliminary on $\kappa$-solutions}

A Riemannian metric $(M,g)$ is called a (gradient) Ricci soliton if there exists a smooth function $f$ on $M$ such that
\begin{equation}\label{soliton-real}
R_{i{j}}-\lambda g_{ij}=\nabla_{i}\nabla_{{j}}f,
\end{equation}
where $R_{i{j}}$ are components of Ricci curvature of $g$, $\nabla$ is a co-derivative associated to $g$ and $\lambda$ is a constant. $(M, g)$
is called shrinking, steady, or expanding according to $\lambda>, =, <0$, respectively.
In case that $(M,J)$ is an $n$-dimensional complex manifold and $g$ is a K\"{a}hler metric, then we call $(M, g)$ a K\"{a}hler-Ricci soliton.
It is easy to see that ($\ref{soliton-real}$) is equivalent to
\begin{equation}\label{soliton-complex}
R_{i\bar{j}}-\lambda g_{i\bar j}=\nabla_{i}\nabla_{\bar{j}}f, \mbox{\quad}\nabla_{\bar{i}}\nabla_{\bar{j}}f=0.
\end{equation}
Let $\varphi_{t}$ and $\psi_{t}$ be the one parameter group generated by vector field $\nabla f$ and $J\nabla f$, respectively. Then
$\varphi_{t}$, $\psi_{t}$ are two families of biholomorphisms of $M$. Moreover $\psi_{t}$ are isometric transformations since $J\nabla f$ is a Killing vector field (cf. \cite{Bry}).

Recall that a complete $n$-dimensional Riemannian manifold $(M^n,g)$ is called $\kappa$-noncollapsed if there exists some $\kappa>0$ such that
${\rm vol}(B(p,r))\geq \kappa r^{n}$ for any $r>0$
whenever $|{\rm Rm}(q)|\leq r^{-2}$ for all $q\in B(p,r)$. For a solution of Ricci flow, Perelman introduced \cite{Pe1},

\begin{defi} Let $g=g(t)$ be a solution of Ricci flow on $M$,
\begin{align}\label{ricci-equation} \frac{\partial g}{\partial t}=-2 {\rm Ric}(g),~t\in(a,b].
\end{align}
We say that $(M,g(t))$ is $\kappa$-noncollapsed on scales at most $r_{0}$ if there exists some $\kappa>0$ such that
$${\rm vol}(B(p,r, t))\geq \kappa r^{n},$$
whenever $|{\rm Rm}(q,t^{\prime})|\leq r^{-2}$\footnote{ It  can be replaced by $|{\rm Rm}(q,t^{\prime})|\leq  C_0 r^{-2}$ for some uniform constant $C_0$.}  for all $q\in B(p,r, t)$,  $t^{\prime}\in(t-r^{2},t]$, $a\leq t-r^{2}$ and $0\leq r\leq r_{0}$.
We say that $(M,g(t))$ is $\kappa$-noncollapsed if it is $\kappa$-noncollapsed on all scales $r_{0}\leq\infty$.
\end{defi}

\begin{defi} A complete solution $(M, g(t))$ of (\ref{ricci-equation}) is called
ancient if it is defined on $(-\infty, 0]$ and the curvature operator of $g(t)$ is bounded and nonnegative for any $t\in (-\infty, 0]$.
A complete K\"{a}hler-Ricci flow $(M, g(t))$ on $t\in (-\infty,0]$ is called ancient if the bisectional curvature of $g(t)$ is bounded
and nonnegative for any $t\in (-\infty, 0]$. Without confusion, we call a $\kappa$-noncollapsed, non-flat ancient solution of (\ref{ricci-equation}) a $\kappa$ (K\"ahler) solution.
\end{defi}

For a complete noncompact Riemannian manifold $(M,g)$ with nonnegative Ricci curvature, we define the asymptotical volume by
$$\mathcal{V}(M,g)=\lim_{r\rightarrow\infty}\frac{{\rm vol}(B(p,r))}{r^{n}}.$$
Clearly, $\mathcal{V}(M,g)$ is independent of the choice of $p$. The following result says that it is always zero for an ancient solution $(M,g(t))$ (cf. \cite{Pe1}, \cite{N}).

\begin{prop}\label{Volume}
Suppose that $(M,g(t))$ is a noncompact and non-flat ancient (K\"{a}hler) solution. Then $\mathcal{V}(M,g(t))=0$ for all $t\leq0$.
\end{prop}

Next, we define the asymptotical scalar curvature of $g$ by
$$\mathcal{R}(M,g)={\rm limsup}_{\rho(p,x)\rightarrow\infty}R(x)\rho^{2}(p, x),$$
where $\rho(p, \cdot)$ is a distance function from a fixed point $p\in M$.
It is easy to see that $\mathcal{R}(M,g)$ is independent of the choice of $p$. By Proposition \ref{Volume}, we prove

\begin{cor}\label{ASR}
The asymptotical scalar curvature $\mathcal{R}(M,g(t))$ of a noncompact $\kappa$ (K\"ahler) solution $(M,g(t))$ is infinite.
\end{cor}

\begin{proof}
We prove the corollary by contradiction. Suppose $\mathcal{R}(M,g(t_0))<A$ for some positive constant $A>1$ and $t_0\leq 0$. For a fixed point $p\in M$,
we have $R(x,t_0)\leq Ar^{-2}$ for all $x\in M\setminus B(p,r, t_0)$ when $r>r_{0}$. Fix any $q\in B(p, 3\sqrt{A}r, t_0)\setminus B(p, 2\sqrt{A}r, t_0)$. Then, we have $R(x,t_0)\leq r^{-2}$ for all $x\in B(q,r, t_0)$.
It follows
 $$|{\rm Rm}(x,t_0)\leq  C_0r^{-2},~\forall ~x\in B(q,r, t_0).$$
Since $(M,g(t_0))$ is $\kappa$-noncollapsed,  we get ${\rm vol} (B(q,r, t_0))\geq \kappa r^{n}$. By the volume comparison theorem,
\begin{align}{\rm vol}(B(p, (3\sqrt{A}+1)r, t_0) &\geq{\rm vol}B(q, r, t_0)\notag\\
&\ge \kappa(3\sqrt{A}+1)^{-n}(3\sqrt{A}+1)r)^{n}, ~\forall ~r>r_{0}.\notag
\end{align}
It follows
$$\mathcal{V}(M,g(t))\geq \kappa(3\sqrt{A}+1)^{-n}.$$
This is a contradiction with  Proposition \ref{Volume}!
\end{proof}

\vskip6mm

\vskip6mm

\section {Perelman's compactness theorem}

In \cite{Pe1}, Perelman proved the following compactness theorem for $3$-dimen-sional $\kappa$-solutions.

\begin{theo}\label{Perelman's-compactness-theorem}
Let $(M_{k},g_{k}(t); p_{k})$ be a sequence of 3-dimensional $\kappa$-solution on a noncompact manifold $M$ with $R(p_{k},0)=1$. Then, $(M_k,g_k(t);p_{k})$ subsequently converge to a $\kappa$-solution.
\end{theo}

To generalize Theorem \ref{Perelman's-compactness-theorem} to higher dimensional $\kappa$ (K\"ahler) solutions, we introduce

\begin{defi}\label{def-pesudo-solution}
We call a $\kappa$-noncollapsed Ricci flow $(M, g(t))$ pseudo $\kappa$ (K\"ahler) solution if it is defined on $ M\times (-\infty,0]$ with nonnegative curvature operator (nonnegative bisectional curvature)
such that the following Harnack inequality holds along the flow:
\begin{equation}\label{eq:Harnack-real}
\frac{\partial R}{\partial t}+2\nabla_{i}RV^{i}+2R_{ij}V^{i}V^{j}\geq0,~\forall~V\in TM
\end{equation}
or in K\"ahler case,
\begin{equation}\label{eq:Harnack-complex}
\frac{\partial R}{\partial t}+\nabla_{i}RV^{i}+\nabla_{\bar{i}}RV^{\bar{i}}+R_{i\bar{j}}V^{i}V^{\bar{j}}\geq0,~\forall~V\in T^{(1,0)}M.
\end{equation}

\end{defi}

(\ref{eq:Harnack-real}) or (\ref{eq:Harnack-complex}) implies the Harnack inequality (cf. \cite{H}, \cite{Ca1}),
\begin{align}\label{harnack-r}\frac{R(x_{2},t_{2})}{R(x_{1},t_{1})}\geq e^{-\frac{d^{2}_{t_{1}}(x_{1},x_{2})}{2(t_{2}-t_{1})}},~\forall~t_1\le t_2.
\end{align}
In this section, we prove

\begin{theo}\label{compactness-theorem}
Let $(M_{k},g_{k}(t); p_{k})$ be a sequence of $n$-dimensional $\kappa$ (K\"ahler) solution on a noncompact manifold with $R(p_{k},0)=1$. Then $(M_k,g_{k}(t);p_{k})$ subsequently converge to a pseudo $\kappa$ (K\"ahler) solution of Ricci flow.
\end{theo}

 It was mentioned by Morgan and  Tian that Perelman's argument still works for higher dimensional $\kappa$-solutions \cite[p. 222]{MT} (also see \cite[Theorem 20.9]{N}).
In the following,  we outline a  proof of  Theorem \ref{compactness-theorem}  from one of Theorem 9.64 in \cite{MT} for Perelman's Theorem \ref {Perelman's-compactness-theorem} with several technical lemmas, some of which   will be  used in Section 4, 5 (also see Proposition \ref{estimate-linear-decay}).  First we need an elementary lemma (cf. \cite{MT}).

\begin{lem}\label{lem-constructing-sequence}
Let $(M,g)$ be a Riemannian manifold and $p\in M$. Let $f$ a continuous and bounded function defined on $B(p,2r)\rightarrow\mathbb{R}$ with $f(p)>0$. Then there is a point $q\in B(p,2r)$ such that $f(q)\geq f(p)$, $d(p,q)\leq2r(1-\alpha)$ and $f(q^{\prime})<2f(q)$ for all $q^{\prime}\in B(q,\alpha r)$, where $\alpha=f(p)/f(q)$.
\end{lem}

By Proposition \ref{Volume} and Lemma \ref{lem-constructing-sequence}, we prove

\begin{lem}\label{cor of Volume}
Let $(M_{k},g_{k}(t), p_k)$ be a sequence of n-dimensional ancient solutions of flow (\ref{ricci-equation}). Let $\nu>0$. Suppose
that there are $p_{k}\in M_{k}$ and $r_{k}>0$ such that ${\rm vol}(B(p_{k},r_{k}, 0))\geq \nu r_{k}^{2n}$. Then there is a $C(\nu)$ independent of $k$ such that $r^{2}_{k}R(q,0)\leq C(\nu)$ for all $q\in B(p_{k},r_{k}, 0)$.
\end{lem}

\begin{proof}
We argue by contradiction. Then there is a sequence of points $q_{k}\in B(p_{k},r_{k}, 0)$ such that $r_{k}^{2}R(q_{k},0)\rightarrow\infty$ as $k\rightarrow\infty$. Let $f(x,t)=\sqrt{R(x,t)}$.
Applying Lemma \ref{lem-constructing-sequence} to $f(x,0)$ defined on $B(q_k, 2r_k, 0)$, we see that there are $q_{k}^{\prime}\in B(q_{k},2r_{k}, 0)$ such that $R(q_{k}^{\prime},0)\geq R(q_{k},0)$ and $R(q,0)\leq 4R(q_{k}^{\prime},0)$ for all $q\in B(q_{k}^{\prime},s_{k}, 0)$ with $s_{k}=r_{k}\sqrt{R(q_{k},0)/R(q_{k}^{\prime},0)}$. Since
$\frac{\partial R}{\partial t}\geq0$
by the Harnack inequality (\ref{eq:Harnack-real}) (or (\ref{eq:Harnack-complex})), we get
\begin{equation}\label{eq:3-1}
R(q,t)\leq 4R(q_{k}^{\prime},0), ~ \forall ~t\leq 0, q\in B(q_{k}^{\prime},s_{k}, 0).
\end{equation}
On the the hand, by the relation
$$\rho_{0}(p_{k},q_{k}^{\prime})\leq \rho_{0}(p_{k},q_{k})+\rho_{0}(q_{k},q_{k}^{\prime})<3r_{k},$$
where $\rho_0(p_k,q_k)$ is a distance function between two points $p_k,q_k$ in $M_k$ with respect to $g_k(0)$, we have
\begin{align} {\rm vol}(B(q_{k}^{\prime},4r_{k}, 0))\geq {\rm vol}(B(p_{k},r_{k}, 0)\geq(\nu/4^{2n})(4 r_{k})^{2n}.\notag
\end{align}
It follows from the Bishop-Gromov volume comparison theorem,
\begin{equation}\label{eq:3-3}{\rm vol}(B(q_{k}^{\prime},s, 0))\geq(\nu/4^{2n})s^{2n},~ \forall~ s\leq s_{k}\leq 3r_{k}.
\end{equation}

Now we consider the rescaled flows $(M_{k},Q_{k}g(Q_{k}^{-1}t);  q_{k}^{\prime})$ with $Q_{k}=R(q_{k}^{\prime},$ $ 0)$. By $(\ref{eq:3-1})$ and $(\ref{eq:3-3})$, we see
that the flows are all $(\nu/4^{2n})$-noncollapsed with the scalar curvature bounded by $4$ on the geodesic balls of radii $s_{k}\sqrt{Q_{k}}$ centered at $q_{k}^{\prime}$.
Since $s_{k}\sqrt{Q_{k}} =r_k\sqrt {R(q_k, 0)}\rightarrow\infty$ as $k\rightarrow\infty$, by the Hamilton's compactness theorem \cite{H4},
$(M_{k},Q_{k}g(Q_{k}^{-1}t);  q_{k}^{\prime})$ converge subsequently to an ancient solution $(M_{\infty},g_{\infty}(t))$. Note that ($\ref{eq:3-3}$) implies the limit $(M_{\infty},g_{\infty}(0))$ has the maximal volume growth.
This is a contradiction  with  Proposition \ref{Volume}.
\end{proof}

\begin{lem}\label{lem-local bound}Let $(M,g(t), p)$ be an n-dimensional $\kappa$-solution of Ricci flow.
Suppose that there exists a point $q\in (M,g(0))$ such that
\begin{align}\label{assumption-r}\rho_{0}(p,q)^{2}R(q,0)=1.\end{align}
Then, there is a uniform constant $C>0$ independent of $g(t)$  such that $R(x,0)/R(q, 0)\leq C$ for all $x\in B(q,2d, 0)$, where $d=\rho_{0}(p,q)$.
\end{lem}

\begin{proof}
Suppose that the lemma is not true. Then there is a sequence of $\kappa$-solutions $(M_{k},g_{k}(t);  p_{k})$ with points $q_{k}^{\prime}\in B(q_{k},2d_{k}, 0)$ such that $$\lim_{k\rightarrow\infty}(2d_{k})^{2}R(q_{k}^{\prime},0)=\infty,$$
where $d_{k}=\rho_{0}(p_{k},q_{k})$ and $\rho_{0}(p_{k},q_{k})^{2}R(q,0)=1$. By Lemma \ref{cor of Volume}, it is easy to see that for any $\nu>0$, there is an $N(\nu)$ such that
$${\rm vol}(B(q_{k},2d_{k}, 0))<\nu(2d_{k})^{2n}, ~\forall ~k>N(\nu).$$
Hence, by taking the diamond method, we may assume that
\begin{equation}\label{eq:1}
\lim_{k\rightarrow\infty}{\rm vol}(B(q_{k},2d_{k}, 0))/(2d_{k})^{2n}=0.
\end{equation}
In particular,
$${\rm vol}(B(q_{k},2d_{k}, 0))<(\omega_{2n}/2)(2d_{k})^{2n}, ~\forall ~k\ge k_0,$$
where $\omega_{2n}$ is the volume of unit ball in $\mathbb{R}^{2n}$.
Therefore, by the Bishop-Gromov volume comparison theorem, there exists a $r_{k}<2d_{k}$ such that
\begin{equation}\label{eq:2}
{\rm vol}(B(q_{k},r_{k}, 0))=(\omega_{2n}/2)r_{k}^{2n}.
\end{equation}
Note that by (\ref{eq:1}) and (\ref{eq:2}) we have $\lim_{k\rightarrow\infty}r_{k}/d_{k}=0$.

Next we consider a sequence of rescaled ancient flows $(M_{k},g^{\prime}_{k}(t); q_{k})$, where $g^{\prime}_{k}(t)=r_{k}^{-2}g_{k}(r_{k}^{2}t)$. Then by (\ref{eq:2}), we have
$${\rm vol}(B(q_{k},1+A, g^{\prime}_{k}(0)))\geq {\rm vol}(B(q_{k},1, g^{\prime}_{k}(0)))=\frac{\omega_{2n}}{2(1+A)^{2n}}(1+A)^{2n},$$
where $A>0$ is any fixed constant.
Applying Lemma \ref{cor of Volume} to the ball $B(q_{k},1+A; g^{\prime}_{k}(0))$,
there is a constant $K(A)$ independent of $k$ such that
$$(1+A)^{2}R(q, {g_{k}^{\prime}}(0))\leq K(A),
~\forall ~q\in B (q_{k},1+ A; g^{\prime}_{k}(0)).$$
Hence by the Harnack inequality, scalar curvature  of  $g^{\prime}_{k}(t)$  on $B_{g^{\prime}_{k}(0)}(q_{k},A, 0)\times(-\infty,0]$ is uniformly bounded by $K(A)$,  and so  is  its sectional curvature.
By the Hamilton's compactness theorem, $(M_{k},g^{\prime}_{k}(t); q_{k})$ converges to a limit flow $(M_{\infty},g_{\infty}(t); q_{\infty})$.
Note by (\ref{assumption-r}) that
$$R(q_{\infty}, g_{\infty}(0))=\lim_{k\rightarrow\infty}R(q_{k}, g_{k}'(0))=\lim_{k\rightarrow\infty}\frac{(r_{k})^{2}}{d_{k}^{2}}=0.$$
Therefore, the strong maximum principle implies that $(M_{\infty},g_{\infty}(t))$ is flat flow.

At last, we prove that $(M_{\infty},g_{\infty}(t))$ is isometric to the Euclidean space for any $t\le 0$. We need to consider at $t=0$. Fix any $r>0$. Obviously,
$$\sup_{x\in B(q_{\infty},r; g^{\prime}_{k}(0))}|{\rm Rm}(x)|=0\le \varepsilon.$$
where $\varepsilon$ can be chosen so that
$\frac{\pi}{\sqrt{\varepsilon}}>2r.$
Note that $(M_{\infty},g_{\infty}(t))$ is $\kappa$-noncollapsed for each $t\le 0$. Thus we have
$${\rm vol}(B (q_{\infty},r; g_{\infty}(0))\geq \kappa r^{n}.$$
By  an estimate of Cheeger-Taylor-Gromov for the injective radius  \cite{CGT}, it follows
$${\rm inj}(q_{\infty})\geq \frac{\pi}{2\sqrt{\varepsilon}}\frac{1}{1+\frac{\omega_{n}(r/4)^{n}}{{\rm vol}(B (q_{\infty},r/4; g_{\infty}(0) ))}}\geq \frac{\kappa}{\kappa+\omega_{n}}\cdot r.$$
Hence $B(q_{\infty},\frac{\kappa}{\kappa+\omega_{n}}\cdot r; g_{\infty}(0))$ is simply connected for all $r>0$. Therefore, $M_{\infty}$ is a simply connected, and consequently $g_\infty(t)$ are all isometric to the Euclidean metric.

The above implies that ${\rm vol}(B(q_{\infty},1; g_{\infty}(0)))=\omega_{n}$. On the other hand, by the convergence of $(M_{k},g_{k}(t);p_{k})$ and the relation (\ref{eq:2}), we get
$${\rm vol}(B (q_{\infty},1; g_{\infty}(0)))=\omega_{n}/2.$$
This is a contradiction. The lemma is proved.
\end{proof}

\begin{lem}\label{lem-curvature bound}Let $(M,g(t); p)$ be a $\kappa$-solution with $R(p,0)=1$.
Then there exists a $\delta>0$ independent of $g(t)$  such that $R(q,0)\leq \delta^{-2}$ for all $q\in B(p,\delta;0)$.
\end{lem}

\begin{proof}
By Corollary \ref{ASR}, there exists point $q\in M$ such that
\begin{align}\label{r-rho-squre} \rho_{0}(p,q)^{2}R(q,0)=1.
\end{align}
Applying Lemma \ref{lem-local bound}, we get
\begin{align}\label{harnack-relation-r} R(x,0)/R(q,0)\leq A,~\forall ~x\in B(q,2d, 0),
\end{align}
where $d= \rho_{0}(p,q)$. It suffices to prove that $R(q,0)\leq C_{0}$ for some $C_{0}>0$.

By  the Harnack inequality, we have
\begin{align}\label{r-on-ball} R(x,t)\leq  R(x,0), ~\forall ~x\in B(q,2d, 0).
\end{align}
Thus the Ricci curvature of $g(t)$ is uniformly bounded by $A R(q,0)$ on $B(q,$ $2d, 0)$  by ( \ref{harnack-relation-r}).
By the flow (\ref{ricci-equation}), it follows
\begin{equation*}
\frac{{\rm d}}{dt}L(t)\geq -AR(q,0)L(t),
\end{equation*}
where $L(t)$ is the length of $\gamma(s)$ with respect to $g(t)$ for any $t\leq 0$ and $\gamma(s)$ is a minimal geodesic connecting $p$ and $q$ with respect to $g(0)$.
Thus
$$ d_{t}(p,q)\le L(t)\leq e^{-AR(q,0)t}L(0)=e^{-AR(q,0)t}R(q,0)^{-1/2}.$$
Choose $t_{c}=-cR^{-1}(q,0)$ , where $0<c<1$ will be determined later.
By the Harnack inequality (\ref{harnack-r}),
$$\frac{R(p,0)}{R(q,t)}\geq e^{\frac{d^{2}_{t}(p,q)}{2t}},$$
we obtain
\begin{align}\label{t-curvature}R(q,t_{c})\leq \exp(e^{2cA}/2c)\leq e^{C^{\prime}/2c}.
\end{align}

Let $\widetilde{g}(t)=R(q,0)g(R(q,0)^{-1}t)$. By (\ref{r-on-ball}),
$$|\widetilde{R}(x,t)|\le A,~\forall ~x\in ~B(q,2d,0),~t\le 0.$$
Thus $|\widetilde{{\rm Rm}}(x,t)|\le A'$ for any $ x\in ~B(q,2d,0)$ and $t\le 0.$
Since the Ricci curvature is nonnegative,
$$\widetilde{B}(q,2,t)\subseteq \widetilde{B}(q,2,0)=B(q,2d,0),~\forall~t\le0.$$
By the Shi's higher order estimates for curvature tensors \cite{S1}, we have
$$|\widetilde{\Delta} \widetilde{R}|(x,t)\le C(A),~\forall x\in \widetilde{B}(q,1,-1)~,t\in(-1,0].$$
It follows
$$|\Delta R|(x,t)\le C R^2(q,0),~\forall~ x\in \widetilde{B}(q,1,-1), ~t\in(-R(q,0)^{-1},0].$$
Hence
\begin{align}\label{shi-estimate}
|\Delta R(q,t)|\le C R^2(q,0),~\forall ~t\in(-R(q,0)^{-1},0].
\end{align}

By (\ref{shi-estimate}) and  the equation
$$\frac{\partial R}{\partial t}=\triangle R+2|{\rm Ric}|^{2},$$
we  have
$$|\frac{\partial }{\partial t}R(q,t)|\leq C'R^{2}(q,0).$$
By (\ref{t-curvature}),  it follows
$$R(q,0)\leq R(q,t_{c})+ C'|t_{c}| R^{2}(q,0)^2\leq e^{C/2c}+cC'R(q,0).$$
Thus by choosing $c=\frac{1}{2}(C')^{-1}$, we derive
$$R(q,0)\le C_0.$$
Let $\delta=\sqrt{(AC_{0})^{-1}}$. Then the lemma follows from (\ref{r-rho-squre}) and (\ref{harnack-relation-r}) immediately.

\end{proof}

\begin{proof}[Proof of Theorem \ref{compactness-theorem}] By Lemma \ref{lem-curvature bound},
the $\kappa$-noncollapsed condition of $(M_{k},$ $g_{k}(t))$ implies
$${\rm vol}(B(p_{k},\delta, 0))\geq \kappa \delta^{2n},$$
where $\delta>0$ is a uniform number.
By the Bishop-Gromov volume comparison theorem, we have
$${\rm vol} (B(p_{k},\delta+r, 0))\geq {\rm vol} (B(p_{k},\delta, 0))\geq \frac{\kappa}{(1+(r/\delta))^{2n}}(\delta+r)^{2n},~\forall~r>0.$$
Applying Lemma \ref{cor of Volume} to each ball $B(p_{k},\delta+r, 0 )$, we see that there is a $C(r)$ independent of $k$ such that
$$R(q,0)\leq C(r)(r+\delta)^{-2},~\forall ~ q\in B(p_{k},\delta+r, 0).$$
By the Harnack inequality, we also get
$$R(q,t)\leq C(r)(r+\delta)^{-2},~\forall ~ q\in B(p_{k},\delta+r, 0).$$
As a consequence, ${\rm Rm}(q,t)\leq C'(r)(r+\delta)^{-2}$  for any $ q\in B(p_{k},\delta+r, 0).$
Hence, the Hamilton's compactness theorem implies that $(M_k,g_{k}(t); p_{k})$ subsequently converge to a limit Ricci flow $(M_\infty, g_\infty(t))$ with nonnegative curvature operator ( or nonnegative bisectional curvature) for any $t\le 0$.
Moreover,
$g_\infty(t)$ satisfies the Harnack inequality (\ref{eq:Harnack-real}) or (\ref{eq:Harnack-complex}) since $g_k(t)$ satisfies the corresponding Harnack inequality (cf. \cite{H}, \cite{Ca1}).

\end{proof}

By using the argument in the proof of Theorem  \ref{compactness-theorem},  we  have the following pointwisely  estiamte for the Lapalce of scalar curvature.

\begin{prop}\label{estimate-linear-decay}
Let $(M,g(t))$ be a $\kappa$-solution. Then there is a constant $C$ independent of $p,t$ such that
$$\frac{|\Delta R(p,t)|}{R^{2}(p,t)}\leq C, \mbox{\quad}\forall~ (p,t)\in M\times(-\infty,0].$$
\end{prop}

\begin{proof}On the contrary,  we can find a sequence of $p_{i}$ and $t_{i}$ such that
\begin{equation}\label{eq:contradiction}
\lim_{i\rightarrow\infty}\frac{|\Delta R(p_{i},t_{i})|}{R^{2}(p_{i},t_{i})}=\infty.
\end{equation}
Consider a sequence of rescaled flows $(M,g_{i}(t); p_{i})$ with
$$g_{i}(t)=R(p_{i},t_{i})g(R^{-1}(p_{i},t_{i})t+t_{i}).$$
Then $R(p_i, g_i(0))=1$. As in the proof of Theorem \ref{compactness-theorem},
we see that there is a constant $C$ independent of $i$ such that
$$R(q, g_{i}(0))\leq C,~\forall ~q\in B(p_{i},1; g_{i}(0)).$$
Since   $\frac{\partial}{\partial t}R(q, g_{i}(t))\geq 0$ by  the Harnack inequality,
$$R(q, g_{i}(t))\leq C, ~\forall ~(q,t) \in B(p_{i},1; g_{i}(0))\times[-1,0].$$
As a consequence,
$$|{\rm Rm}(q, g_{i}(t))|
\le   C', ~\forall ~(q,t) \in B(p_{i},1; g_{i}(0))\times[-1,0].$$
On the other hand,  Ricci curvature  of  $g_{i}(t)$ is nonnegative along the flow, so   $g_{i}(t)$ is decreasing along the flow. Then
$$B(p_{i},1; g_{i}(-1))\subset B(p_{i},1; g_{i}(0)).$$
Hence
$$|{\rm Rm}(q, g_{i}(t))|\le C',~\forall~(q,t)\in B(p_{i},1; g_{i}(-1))\times[-1,0].$$
By the Shi's higher order estimate, we get 
$$|\Delta R(q, g_{i}(t))|\leq C'',~\forall ~q\in B(p_{i},\frac{1}{2}; g_{i}(-1))\times[-\frac{1}{2},0].$$
In particular,
$$\frac{|\Delta R(p_{i},t_{i})|}{R^{2}(p_{i},t_{i})}=|\Delta R (p_{i},g_{i}(0))|\leq C''.$$
This is a contradiction  with (\ref{eq:contradiction}).
\end{proof}

Proposition \ref{estimate-linear-decay} will be used in the proof of Theorem \ref{soliton-1} next section.

\section {Asymptotical geometry of solitons}

In this section, we use Theorem \ref{compactness-theorem} to prove Theorem \ref{soliton-1}. Let $\phi_{t}$ be a family of biholomorphisms generated by $-\nabla f$. Let $g(t)=\phi^{*}_{t}(g)$. Then $g(t)$
satisfies the Ricci flow (\ref{ricci-equation}).
In \cite{DZ}, the authors proved that there exists a unique equilibrium point $o$ such that $\nabla f(o)=0$ for a steady gradient K\"{a}hler-Ricci soliton with nonnegative bisectional curvature and positive
Ricci curvature. Thus for any $p\in M\setminus \{o\}$, it is easy to see that $\phi_{t}(p)$ converge to $o$ as $t\rightarrow\infty$. In the following,
we show that the growth order of $\rho(o, \phi_{t}(p))$ is actually equivalent to $|t|$ as $t\rightarrow-\infty$.

\begin{lem}\label{lem-equivalence}
Let $o$ be the equilibrium point as above. Then for any $p\in M\setminus \{o\}$, there exist constants $C_{1},C_{2}>0$
and $t_{0}\leq0$ such that
\begin{align}\label{distance-in-t} C_{1}|t|\leq \rho(o,\phi_{t}(p))\leq C_{2}|t|,~ \forall ~t\leq t_{0}.
\end{align}
\end{lem}

\begin{proof} By the identity (cf. \cite{H1}),
\begin{align}\label{scalar-curvature-identity}
R+|\nabla f|^{2}=A_0,
\end{align}
where $A_0$ is a constant, we have
$$|\nabla f|^{2}(x)+R(x)=R(o),\mbox{\quad}\forall ~x~\in M.$$
Then
$$\frac{{\rm d}}{{\rm d}t }R(\phi_{t}(p))={\rm Ric}(\nabla f,\overline{\nabla f})\ge0,\mbox{\quad}\forall~ t~\leq0.$$
In particular,
$$0\le R(\phi_{t}(p))\le R(p),~\forall t~\leq0.$$
Since
$$\frac{{\rm d}}{{\rm d}t }f(\phi_{t}(p))=-|\nabla f|^{2}(\phi_{t}(p)),\mbox{\quad}\forall~ t~\leq0,$$
we get from (\ref{scalar-curvature-identity}),
$$ R(o)-R(p)\leq -\frac{{\rm d}}{{\rm d}t }f(\phi_{t}(p))\leq R(o),~\forall t~\leq0.$$
It follows
\begin{align}\label{f-time-1}
(R(o)-R(p))|t|\leq f(p)-f(\phi_{t}(p))\leq R(o)|t|,\mbox{\quad}\forall~ t~ \leq0.
\end{align}
Consequently,
$$(R(o)-R(p))|t|+C(p)\leq f(o)-f(\phi_{t}(p))\leq R(o)|t|+C(p),\mbox{\quad}\forall ~t\leq0,$$
where $C(p)=f(o)-f(p)$.
On the there hand, by Proposition 7 in \cite{CaCh}, there are constants $C_{1},C_{2}>0$ such that
\begin{align}\label{f-distance-1}
 C_{1}\rho (o, \phi_{t}(p))\leq f(o)-f(\phi_{t}(p))\leq C_{2}\rho (o, \phi_{t}(p)),~\forall~ t\leq t_{0},
 \end{align}
where  $t_{0}$ is small enough constant. Combining the above two inequalities, we obtain (\ref{distance-in-t}).
\end{proof}

\begin{rem}
Let $A(r)=\{p\in M: f(p)=r \}$ for any $r\in \mathbb{R}$. Then $A(r)$ is compact as $r>>1$ since $f$ is strictly convex. Thus from the proof of Lemma \ref{lem-equivalence}, the constants $C_{1}$ and $C_{2}$ in (\ref{distance-in-t})
can be chosen uniformly for all $p\in A(r)$ so that both of them are independent of $t$.
\end{rem}

Combining Lemma \ref{lem-equivalence} and Proposition \ref{estimate-linear-decay}, we obtain a lower bound growth estimate for scalar curvature.

\begin{prop}\label{curvature-lower-estimate}For a $\kappa$-noncollapsed steady K\"ahler-Ricci soliton $(M,g)$ with nonnegative bisectional curvature and positive
Ricci curvature, the scalar curvature satisfies
\begin{align}\label{lower-bound-r}\frac{C}{\rho(x)}\leq R(x),~{\rm if}~\rho(x)\ge r_0,
\end{align}
where $\rho(x)=\rho(o,x)$ and $C>0$ is a uniform constant.
\end{prop}

\begin{proof}
Since the scalar curvature $R(p,t)$ of $g(p,t)$ satisfies
$$\frac{\partial}{\partial t}R(p,t)=\Delta R(p,t)+2|{\rm Ric}(p,t)|^{2}, $$
by Proposition \ref{estimate-linear-decay}, there is a positive constant $C>0$ such that
$$|\frac{\partial}{\partial t}R^{-1}(p,t)|\leq \frac{|\Delta R(p,t)|}{R^{2}(p,t)}+\frac{2|{\rm Ric}(p,t)|^{2}}{R^{2}(p,t)}\leq C+2, $$
and consequently,
\begin{equation}\label{decay-along-integral-curve}
R(p,t)|t|\geq \frac{|t|}{(C+2)|t|+R(p,0)^{-1}}\geq\frac{1}{2(C+2)}
\end{equation}
as long as $|t|$ is large enough.

Next we show that (\ref{decay-along-integral-curve}) implies  (\ref{lower-bound-r}). We may assume $f(o)=0$. For any $x$ such that $f(x)\gg 1$. Note that there exists  $p_x\in \{q\in M|f(q)=1\}$ and $t_x<0$ such that $\phi_{t_x}(p_x)=x$. By  (\ref{decay-along-integral-curve}) together with (\ref{f-time-1}) and(\ref{f-distance-1}), we have
\begin{align*}
R(x)\ge&\frac{1}{|t_x|}\cdot\frac{1}{(C+2)+(R(p_x)|t_x|)^{-1}}\\
\ge&\frac{R(o)-R(p_x)}{f(x)-f(p_x)}\cdot\frac{1}{(C+2)+(R(p_x)|t_x|)^{-1}}\\
\ge&\frac{R(o)-R(p_x)}{2(f(x)-f(o))}\cdot\frac{1}{(C+2)+(R(p_x)|t_x|)^{-1}}\\
\ge&\frac{R(o)-M_1}{2C_2\rho(x)}\cdot\frac{1}{2(C+2)},~\forall ~|t_x|\ge \frac{C+2}{R(p_x)}.
\end{align*}
Here $M_1=\sup_{q\in\{f=1\}}R(q)$.
On the other hand, by (\ref{f-time-1}), we have
$$|t_x|\ge \frac{f(x)-f(p_x)}{R(o)-R(p_x)}=\frac{f(x)-1}{R(o)-R(p_x)}.$$
Then   it holds
$$R(x)\ge \frac{R(o) -M_1}{4C_2(C+2)\rho(x)}\ge  \frac{1}{C_3(C+2)\rho(x)} ,$$
 as long as $f(x)\ge \frac{C+2}{m_1}\cdot(R(o)-m_1)+1,$
 where  $m_1=\inf_{q\in\{f=1\}}R(q)$.
Note that $C$, $C_3$ and $m_1$  are all independent of $x,t$. Hence, by (\ref{f-distance-1}),  we get (\ref{lower-bound-r}).

\end{proof}

Now we are ready to prove Theorem \ref{soliton-1}.

\begin{proof}[Proof Theorem \ref{soliton-1}]    By Proposition \ref{curvature-lower-estimate}, we have
\begin{equation}\label{eq:6}
\lim_{i\to\infty}\rho^{2}(o,p_i)R(p_i,0)=\infty.
\end{equation}
Let $\hat g_i( t)=R(p_{i},0)g(R^{-1}(p_{i},0)t)$ be a sequence of rescaled Ricci flows of $g(t)$. Clearly, $R(p_i; \hat g_i(0)))=1.$
Then applying Theorem \ref{compactness-theorem} to $(M, \hat g_i(t); p_{i}),$ we see that $(M, \hat g_i( t); p_i)$ converges to a pseudo $\kappa$ K\"ahler solution $(M_{\infty}, \tilde g(t),$ $ p_{\infty})$ of (\ref{ricci-equation}).
Moreover, by (\ref{eq:6}) and nonnegative sectional curvature condition, we can construct a geodesic line through $p_\infty$ in $(M_{\infty}, \tilde g(t);  p_{\infty})$ (cf. Theorem 5.35 in \cite{MT}). Thus by the Cheeger-Gromoll splitting theorem \cite{CG},
$(M_{\infty},\tilde g(0))$ must split off a line. Let $X$ be the vector field tangent to the line with the norm equal to $1$ and $J_{\infty}$ the complex structure on $M_{\infty}$. Then
$J_{\infty}X$ generates a geodesic curve $\gamma(s)$ in $M_\infty$. If $\gamma(s)$ is not closed, it is a geodesic line on $M_{\infty}$. If $\gamma(s)$ is closed, it is a flat $\mathbb S^1$.
Hence $(M_{\infty}, \tilde g( 0))$ splits off a complex line $N_1=\mathbb C^1$ or a cylinder $N_1=\mathbb R^1\times \mathbb S^1$. Namely, $M_\infty=N_1\times N_2$ and
$\tilde g( t)={\rm d}z\otimes{\rm d}\overline{z}+g_{N_{2}}(t)$, where $g_{N_2}(t)$ is a pseudo $\kappa$ K\"ahler solution of (\ref{ricci-equation}) on a complex manifold $N_2$
with dimension $n-1$.

In case ${\rm dim}_{\mathbb C}(M)=2$, $(M_{\infty}, \tilde g( t))= (N_{1}\times N_{2},{\rm d}z\otimes{\rm d}\overline{z}+ g_{N_{2}}(t))$, where
$g_{N_{2}}$ is a pseudo $\kappa$ K\"ahler solution of (\ref{ricci-equation}) on a surface $N_2$. In particular, the scalar curvture $\tilde R(\cdot,t)$ of $g_{N_{2}}(t))$ satisfies
Harnack inequality
\begin{align}\label{harnack-r-2} \frac{\partial }{\partial t}\tilde R(\cdot,t)\geq0, ~{\rm in}~ N_{2}\times (-\infty, 0].
\end{align}
By Lemma \ref{round-sphere} below, we see that $(N_{2},g_{N_{2}}(t))=(\mathbb{C}\mathbb{P}^{1},(1-t)g_{FS})$.
\end{proof}

Since Theorem \ref{compactness-theorem} holds for $\kappa$-solutions and all lemmas in this section are true for all steady  Ricci solitons, one can prove Theorem \ref{theorem-soliton-real} by the same argument as in the proof of Theorem \ref{soliton-1}.

The following lemma is a generalization of Corollary 11.3 in \cite{Pe1} which says: Any oriented $\kappa$-solution on a surface is a shrinking round sphere.
\begin{lem}\label{round-sphere}
Any oriented pseudo $\kappa$-solution $(M,g(\cdot, t))$ $(t\le 0)$ on a surface is a shrinking round sphere.
\end{lem}

\begin{proof}
By Corollary 11.3 in \cite{Pe1}, it suffices to rule out  the case that $(M,g(t))$ is noncompact and has unbounded curvature.
In this case, we may assume that there is a sequence of points $p_{i}$ such that $R(p_{i},-1)\rightarrow\infty$ and $\rho_{g(-1)}( p_{0}, p_{i})\rightarrow\infty$, where $p_{0}$ is a fixed point. In particular,
\begin{equation}\label{eq:5}
\rho^{2}_{g(-1)}( p_{0}, p_{i} )R(p_{i},-1)\rightarrow\infty,~{\rm as}~ i\rightarrow\infty.
\end{equation}
By taking $f(x,t)=\sqrt{R(x,t)}$ and $r=r_{i}=\frac{1}{4}\rho_{g(-1)}(p_0,p_{i})$ in Lemma \ref{lem-constructing-sequence},
we can find a sequence of points $q_{i}$ such that $R(q_{i},-1)\geq R(p_{i},-1)$ and
$$R(q,-1)\leq 4R(q_{i},-1), ~\forall ~q\in B(q_{i},d_{i}, -1),$$
where $d_{i}\sqrt{R(q_{i},-1)}=r_{i}\sqrt{R(p_{i},-1)}$. Moreover,
$$\rho_{g(-1)}(p_{i},q_{i})\leq2r_{i}=\frac{1}{2}\rho_{g(-1)}(p_{0},p_{i}).$$
Hence
$$\rho_{g(-1)}(p_{0},q_{i})\geq \rho_{g(-1)}(p_{0},p_{i})-\rho_{g(-1)}(p_{i},q_{i})\geq\frac{1}{2}\rho_{g(-1)}(p_{0},p_{i}).$$
It follows
\begin{equation}\label{eq:4}
\lim_{i\rightarrow\infty}\rho^{2}_{g(-1)}(p_{0},q_{i})R(q_{i},-1)=\infty.
\end{equation}

Now, we consider a sequence of rescaled Ricci flows $(M_{i},g_{i}^{\prime}(t); q_{i})$, where $g^{\prime}_{i}(t)=R(q_{i},-1)g(R^{-1}(q_{i},-1)(t+1)-1)$. Since $\frac{\partial}{\partial t}R\geq0$, we have
$$R_{g_{i}^{\prime}}(q,t)\leq 4, ~\forall ~q\in B(q_{i},r_{i}\sqrt{R(p_{i},-1)}, g_{i}^{\prime}), t\leq-1.$$
Note that $ r_{i}\sqrt{R(p_{i},-1)}$ go to infinity as $i\to \infty$ by (\ref{eq:5}). This means that the curvature of flows are locally uniformly bounded. Together with the $\kappa$-noncollapsed condition, $(M_{i},g_{i}^{\prime}(t); q_{i})$ converge
to a limit Ricci flow $(M_{\infty},g_{\infty}(t); q_{\infty})$ for $t\leq-1$. Moreover it is a pseudo $\kappa$ K\"ahler solution. On the other hand, by (\ref{eq:4}) and nonnegative sectional curvature condition, one can construct a geodesic line through $q_\infty$ in $(M_\infty, g_{\infty}; q_{\infty})$ (cf. Theorem 5.35 in \cite{MT}). Thus $(M_{\infty},g_{\infty}(-1)))$ splits off a line. As a consequence, it is isometric to $ \mathbb C^1$ or $\mathbb R^1\times \mathbb S^1$ with the flat metric.
But this is impossible since $R(q_{\infty},-1)=1$. The lemma is proved.
\end{proof}

As an application of Theorem \ref{soliton-1}, we get the following precise estimate for scalar curvature of steady solitons on a complex surface.

\begin{cor}\label{cor-of-splitting-theorem} Let $(M,g,f)$ be a $2$-dimensional
$\kappa$-noncollapsed steady K\"{a}h-$ $ler-Ricci soliton with positive sectional curvature. Let $o\in M$ be the unique equilibrium point
such that $\nabla f(o)=0$ and $p\neq o$. Then
\begin{equation}\label{decay-scalar-t}
R(p,t)|t|\rightarrow1, ~ {\rm as} ~t\rightarrow-\infty.
\end{equation}
As a consequence, there are constants $C_1$ and $C_2$ such that
\begin{align}
\frac{C_1}{\rho(x)}\le R(x)\le \frac{C_2}{\rho(x)}.
\end{align}

\end{cor}

\begin{proof} We first prove the following claim.
\begin{claim}\label{uniform-estimate-for-decay}
\begin{align}
\lim_{t\rightarrow \infty}\frac{\partial }{\partial t}R^{-1}(p,-t)=1.
\end{align}
Moreover, the convergence is uniform for all $p\in A(1)$, where $A(1)=\{q\in M|~ f(q)=1\}$.
\end{claim}
\begin{proof}[Proof of claim]We prove the claim by contradiction. On the contrary,  we can find $\delta>0$, $p_{(i)}\in A(1)$ and $t_{i}\rightarrow\infty$ such that
\begin{align}\label{contr-assump}|\frac{\partial }{\partial t}R^{-1}(p_{(i)},-t_i)-1|\ge \delta>0.
\end{align}
Let $\phi_{t}$ be the group of biholomorphisms generated by $-\nabla f$ and $g(t)$ the corresponding K\"ahler-Ricci flow. Let $p_{i}=\phi_{t_{i}}(p_{(i)})$. Consider a sequence of  rescaled Ricci flows $(M, \hat g_i(t); p_i) $ as in Theorem \ref{soliton-1}, where $\hat g_i(t)=R(p_{i},0)$
 \newline $g(R^{-1}(p_{i},0)t)$. Then $(M, \hat g_i(t); p_i)$ subsequently converges  to a limit Ricci flow $(M_{\infty}, \tilde g( t); p_{\infty})$ while $(M_{\infty}, \tilde g( 0);  p_{\infty})$ is isometric to $(N_{1}\times\mathbb{CP}^{1}, dz\otimes d\bar z+g_{FS})$.
Moreover, by the flow equation for scalar curvature $\tilde R(\cdot, t)$ of $\tilde g( t)$ at $(p_\infty, 0)$,
\begin{align}
\frac{\partial}{\partial t}\tilde R(p_{\infty},0)=\Delta \tilde R(p_{\infty},0)+ 2|\tilde {\rm Ric}|^{2}(p_{\infty},0),\notag
\end{align}
we get
$$\frac{\partial}{\partial t} \tilde R(p_{\infty},0)=1.$$
On the other hand, by the convergence of $(R(p_{i},0)g(R^{-1}(p_{i},0)t;  p_i)$, we have
\begin{align}
\frac{\partial}{\partial t}\tilde R(p_{\infty},0)=\lim_{i\rightarrow\infty} \frac{1}{R^{2}(p_{i},0)}\frac{\partial }{\partial t} R(p_{i},0)=\lim_{i\rightarrow\infty} \frac{1}{R^{2}(p_{(i)},-t_{i})}\frac{\partial}{\partial t} R(p_{(i)},-t_{i}).\notag
\end{align}
Thus
\begin{align}
\lim_{i\rightarrow\infty}G(p_{(i)},t_{i})=1,\notag
\end{align}
where $G(p,t)=\frac{\partial}{\partial t}R^{-1}(p,-t)$. This is   contradict to (\ref{contr-assump}). Hence the claim is true.
\end{proof}

   By Claim \ref{uniform-estimate-for-decay}, for any $\epsilon>0$,  there exists  a $t(\epsilon)<0$ such that
\begin{equation}\label{decay-along-integral-curve-2}
R(p,t)|t|\le \frac{1}{1-\epsilon},~\forall ~p\in A(1),~t\le t(\epsilon).
\end{equation}
We may assume $f(o)=0$. For any $x$ such that $f(x)\gg 1$, we can find  $p_x\in \{q\in M|f(q)=1\}$ and $t_x<0$ such that $\phi_{t_x}(p_x)=x$. By (\ref{decay-along-integral-curve-2}) together with (\ref{f-time-1}) and (\ref{f-distance-1}), we have
\begin{align*}
R(x)\le&\frac{R(o)}{f(x)-f(p_x)}\cdot\frac{1}{1-\epsilon}\\
\le&\frac{2R(o)}{f(x)-f(o)}\cdot\frac{1}{1-\epsilon}\\
\le&\frac{2R(o)}{C\rho(x)}\cdot\frac{1}{1-\epsilon},~\forall ~|t_x|\ge|t(\epsilon)|.
\end{align*}
Note that by (\ref{f-time-1}) we have
$$|t_x|\ge \frac{f(x)-f(p_x)}{R(o)-R(p_x)}=\frac{f(x)-1}{R(o)-R(p_x)}.$$
Thus as long as $f(x)\ge |t(\epsilon)|\cdot(R(o)-m_1)+1$, where  $m_1=\inf_{q\in\{f=1\}}R(q)$.  we obtain
$$R(x)\le \frac{2R(o)}{C\rho(x)}\cdot\frac{1}{1-\epsilon}.$$
The proof is finished.

\end{proof}
\section{Nonexistence of noncollapsed steady K\"{a}hler-Ricci soliton}

In this section, we prove Theorem \ref{main-theorem-nonexistence-1} and Theorem \ref{theorem-eternal-flow}.
First we recall a result of Bryant about the existence of global Poincar\'{e} coordinates on a steady K\"{a}hler-Ricci soliton \cite{Bry}.

\begin{theo}\label{Bryant}
Let $(M,g,f)$ be a steady K\"{a}hler-Ricci soliton with positive Ricci curvature, which admits an equilibrium point on $M$. Let $Z=\frac{\nabla f-\sqrt{-1}J\nabla f}{2}$. Then there exist global holomorphic coordinates
(Poincar\'{e} coordinates)
$z:M\rightarrow \mathbb{C}^{n}$
which linearize Z. Namely,
there are positive constants $h_{1},\cdots,h_{n}$ such that
\begin{equation}\label{Bryant-eq}
Z=\sum_{i=1}^{n}h_{i}z_{i}\frac{\partial}{\partial z_{i}}.
\end{equation}
\end{theo}

\begin{cor}\label{special sequence of point}
Let $(M,g,f)$ be a steady K\"{a}hler-Ricci soliton with nonnegative bisectional curvature and positive Ricci curvature. Then, there exists a sequence of point $p_{k}\rightarrow\infty$ such that every integral curve $\gamma_{k}(s)$ of $J\nabla f$ starting from $p_{k}$ is closed with the same period time. Moreover, the length of $\gamma_{k}(s)$ is uniformly bounded from above.
\end{cor}
\begin{proof}
By Theorem 1.1 in \cite{DZ}, there exists a unique equilibrium point on $M$. According to Theorem \ref{Bryant},
we see that
there exist global Poincar\'{e} coordinates $(z_{1},\cdots,z_{n})$ on $M$ such that $Z=\frac{\nabla f-\sqrt{-1}J\nabla f}{2}$ satisfies (\ref{Bryant-eq}).

Let $z_{i}=x_{i}+\sqrt{-1}y_{i}$. Then
$$J\nabla f=\sum_{i=1}^{n}h_{i}(x_{i}\frac{\partial}{\partial y_{i}}-y_{i}\frac{\partial}{\partial x_{i}}).$$
Choose points $p_{k}=(k,0,\cdots,0,0\cdots,0)\in M$. Then the integral curves of $J\nabla f$ starting from $p_{k}$ are given by
$$\gamma_{k}(s)=(k\cos(h_{1}s),k\sin(h_{1}s),0,\cdots,0,0\cdots,0).$$
Clearly, these curves are all closed with period time $\frac{2\pi}{h_{1}}$.
By the identity (\ref{scalar-curvature-identity}),
$$|\gamma'_k(s)|=|\nabla f|(\gamma_k(s))\leq A_0^{\frac{1}{2}},~{\rm as}~k\to\infty.$$
Hence the length $l_{k}$ of $\gamma_k(s)$ has a uniformly upper bound.
\begin{align}\label{length-gamma-1}l_{k}=\int_{0}^{\frac{2\pi}{h_{1}}}|\gamma'_k(s)|{\rm d}s\leq A_0^{\frac{1}{2}}\frac{2\pi}{h_{1}}.
\end{align}
\end{proof}

In the remaining of this section, we use the estimates in Section 4 to get a lower bound of $l_k$ to derive a contradiction. First, we need the following fundamental lemma.

\begin{lem}\label{lower bound of length}
Let $B(p,r)$ be a geodesic ball with radius $r$ centered at $p$ in a Riemannian manifold $(M,g)$, and $X$ a smooth vector field such that $|X|_{g}(x)\ge C_{0}$ and $|\nabla X|(x)\leq C$ for any $x\in B(p,r)$, where $C$ is a positive constant independent of $x\in B(p,r)$. Let $\gamma(s)$ be the integral curve of $X$ starting from $p$ and we assume that $\gamma(s)$ stays in $B(p,r)$ for all $s\in[0,\infty)$. Then
there exists $c_{0}>0$, which depend only on $r,C,C_{0}$ and the metric $g$ on $B(p,r)$, such that $\gamma(s)$ is away from $p$ for all $s\in(0,c_{0}]$ and
\begin{align}\label{length-gamma-2}{\rm Length}(\gamma(s))\ge c_0 C_0.\end{align}

\end{lem}

\begin{proof}
Suppose that $r_{p}$ is the injective radius at $p\in M$. Set $r_{0}=\min\{r_{p},\frac{r}{2}\}$. By the exponential map, we can choose a normal coordinate $(x_{1},\cdots,x_{n})$ on $B(p,r_{0})$.
Let $X(p)=(X_{1}(p),X_{2}(p),\cdots,X_{n}(p))$. We may assume that $|X_{k}(p)|=\max_{1\le i\le n}|X_{i}(p)|$.
Then $|X_{k}(p)|\geq\frac{C_{0}}{\sqrt{n}}$.  Note that
\begin{align}
x_{k}(\gamma(s))-x_{k}(\gamma(0))=&\frac{dx_{k}(\gamma(0))}{ds}\cdot s+\frac{d^{2}x_{k}(\gamma(\theta s)) }{ds^{2}}\cdot s^{2}\notag\\
=&s(X_{k}(p)+\frac{d^{2}x_{k}(\gamma(\theta s))}{ds^{2}}\cdot s)\notag
\end{align}
and
\begin{align}
|\frac{d^{2}x_{k}(\gamma( s))}{ds^{2}}|=&|\nabla_{X}X-\Gamma^{k}_{ij}\frac{dx_{i}(\gamma(s)}{ds}\frac{dx_{j}(\gamma(s))}{ds}|\notag\\
\le& C_{1}|\nabla_{X}X|_{g}+C_{2}|X(\gamma(s))|_{g}\cdot \max_{1\le i,j,k\le n,x\in B(x,r)}|\Gamma^{k}_{ij}(x)|\notag\\
\le& C_{3},\notag
\end{align}
where $C_{3}$ is independent of $s\in [0,r_{0}]$. Choose $c_{0}=\min \{r_{0},\frac{C_{0}}{2\sqrt{n}C_{3}}\}$. Then
\begin{align}
|X_{k}(p)+\frac{d^{2}x_{k}(\gamma(\theta s))}{ds^{2}}\cdot s|\ge \frac{1}{2}|X_{k}(p)|>0,~\forall~s\in (0,c_{0}].
\end{align}
It follows
\begin{align}
|x_{k}(\gamma(s))-x_{k}(\gamma(0))|\geq \frac{1}{2}s|X_{k}(p)|>0,~\forall~s\in (0,c_{0}].\notag
\end{align}
(\ref{length-gamma-2}) is clear.
Hence, the lemma is proved.
\end{proof}

By Lemma \ref{lower bound of length}, we prove

\begin{lem}\label{prop-contradiction argument}
Let $(M,g,f)$ be an $n$-dimensional $\kappa$-noncollapsed steady K\"ahler-Ricci soliton with nonnegative sectional curvature and positive Ricci curvature. Let $p_{k}$ be the sequence of points constructed in Corollary \ref{special sequence of point}. Then, there exists a positive constant $C$ such that $R(p_{k})>C$, where $C$ is independent of $p_{k}$.
\end{lem}

\begin{proof}
We use the contradiction argument and suppose $R(p_{k})\rightarrow0$ as $k\rightarrow\infty$. Let $g_{k}(t)=R(p_{k})g(R^{-1}(p_{k})t)$. Then by Theorem \ref{soliton-1}, the sequence of Ricci flows $(M,g_{k}(t);p_{k})$ converge subsequently to a limit flow $(M_\infty, g_\infty(\cdot, t);$
\newline $ p_\infty)$. Fix $r>A_0^{\frac{1}{2}}\frac{2\pi}{h_{1}}$ (cf. Corollary \ref{special sequence of point}). Applying Lemma \ref{lem-curvature bound} to flows $(M, g_{k}(t), p_k)$, there is a positive constant $C=C(r)$ independent of $k$ such that
\begin{equation}\label{eq:5-2}
\frac{R(x)}{R(p_{k})}\leq C, ~\forall ~x\in B_{g_{k}(0)}(p_{k},r).
\end{equation}
Thus
\begin{equation}\label{decay-condition}
R(x)\rightarrow0, ~\forall ~x\in B_{g_{k}(0)}(p_{k},r).
\end{equation}
Moreover, the convergence is uniform for $x\in B_{g_{k}(0)}(p_{k},r)$.

Let $X_{(k)}=R(p_{k})^{-\frac{1}{2}}J\nabla f$. Then
$$|X_{(k)}|^{2}_{g_{k}(0)}(x)=|\nabla f|^{2}(x)=A_{0}-R(x).$$
By the identity (\ref{scalar-curvature-identity}) together with the condition (\ref{decay-condition}),
it follows
$$\lim_{k\rightarrow\infty}\sup_{B_{g_{k}(0)}(p_{k},r) }||X_{(k)}|_{g_{k}(0)}-\sqrt{A_{0}}|=0.$$
By Shi's higher order estimate \cite{S1} and soliton equation $(\ref{soliton-complex})$, we also get
$$\sup_{ B_{g_{k}(0)}(p_{k},r)}|\tilde \nabla^{m}_{(g_{k}(0))}X_{(k)}|_{g_{k}(0)}\leq C(n)\sup_{ B_{g_{k}(0)}(p_{k},r)}|\tilde \nabla^{m-1}_{(g_{k}(0))}{\rm Ric}(g_{k}(0))|_{g_{k}(0)}\le C_1,$$
where $\tilde \nabla$ denotes the connection respect to the rescaled metric $g_{k}(0)$.
As a consequence, the restricted vector field $X_{k}$ on $B_{g_{k}(0)}(p_{k},r)$ converges to
a smooth vector field $X_{\infty}$ on $ B_{g_{\infty}(0)}(p_{\infty},r)\subset M_\infty$ in $C^\infty$-topology.
On the other hand,
\begin{align}\label{eq:5-1}
\tilde \nabla_{(g_{k}(0))J\nabla f}(J\nabla f)=
\nabla_{J\nabla f}(J\nabla f)=-\nabla_{\nabla f}(\nabla f)=\nabla R.
\end{align}
Then
$$|\tilde \nabla_{(g_{k}(0))X_{(k)}}X_{(k)}|_{g_{k}(0)}=\frac{|\nabla R|(x)}{R^{\frac{1}{2}}(p_{k})}.$$
%By Theorem 1.4, $\nabla_{g_{\infty}(0)}R_{g_{\infty}(0)}\equiv0$ since the complex dimension of $M$ is $2$.
Note by (\ref{eq:5-2}) and Shi's higher order estimate,
\begin{align}\label{convergence:1}
\frac{|\nabla R|(x)}{R^{\frac{3}{2}}(p_{k})}\le C^{\prime},~\forall ~x\in B_{g_{k}(0)}(p_{k},r).
\end{align}
Thus we get
\begin{align}\label{convergence:2}
|\tilde\nabla_{(g_{\infty}(0))X_{(\infty)}}X_{(\infty)}|_{ g_{\infty}(0)} =\lim_{k\rightarrow\infty}|\tilde\nabla_{(g_{k}(0))X_{(k)}}X_{(k)}|_{g_{k}(0)}=0,
\end{align}
where the convergence is uniform on $B_{g_{k}(0)}(p_{k},r)$.

By the convergence, there are diffeomorphism $\Phi_{k}:B_{g_{k}(0)}(p_k, r) \rightarrow M_{\infty}$ such that $\Phi_{k}(p_{k})=p_{\infty}$,
$\Phi_{k}(g_{k}(0))\rightarrow g_{\infty}(0)$ and
$$(\Phi_{k})_{\ast}(X_{k})\rightarrow X_{\infty},~as~k\rightarrow\infty.$$
By (\ref{convergence:2}), it follows that
$$|\tilde \nabla_{(g_{\infty}(0))\overline{X}_{(k)}}\overline{X}_{(k)}|_{ g_{\infty}(0)}\rightarrow0,~as~k\rightarrow\infty,$$
where $\overline{X}_{(k)}=(\Phi_{k})_{\ast}(X_{k})$.
Let $\overline{\gamma}_{k}=\Phi_k(\gamma_k)$. Clearly $\overline{\gamma}_{k}\subset B_{g_{\infty}(0)}(p_{\infty},r)$ as long as $k$ is sufficiently large, since $\gamma_{k}\subset B_{g_k(0)}(p_{k},r)$ by the choice of $r$. Then we can apply Lemma \ref{lower bound of length} to $\overline{\gamma}_{k}$ to see that there are constants $c_0, A>0$, which depend only metric $g_{\infty}(0)$ on $B_{g_{\infty}(0)}(p_{\infty},r)$ such that
$${\rm Length}(\overline{\gamma}_{k}, g_{\infty}(0))\ge A$$
and $d(\overline{\gamma}_{k}(s),p_{\infty})>0$ for all $s\in(0, c_0]$. It follows
$${\rm Length}(\gamma_{k}, g_{k}(0))\ge \frac{1}{2}{\rm Length}(\overline{\gamma}_{k},g_{\infty}(0))\ge \frac{1}{2}A$$
and $d(\gamma_{k}(s),p_{k})>0$ for all $s\in(0,c_0]$, as long as $k$ is sufficiently large.
On the other hand, by (\ref{length-gamma-1}), we have
$${\rm Length}(\gamma_{k}, g_{k}(0) )\le\frac{2\pi}{h_{1}}A_{0}^{\frac{1}{2}}R(p_{k})^{\frac{1}{2}}\rightarrow0, ~{\rm as}~ k\to\infty.$$
Hence we get a contradiction! The lemma is proved.
\end{proof}

Combining Lemma \ref{prop-contradiction argument} and Corollary \ref{cor-of-splitting-theorem} in Section 4, we prove Theorem \ref{theorem-eternal-flow} in the surfaces case.

\begin{prop}\label{flat-theorem}
Let $(M,g,f)$ be an $2$-dimensional $\kappa$-noncollapsed steady K\"ahler-Ricci soliton with nonnegative sectional curvature. Then $(M,g)$ is flat.
\end{prop}

\begin{proof}
If $(M,g)$ is compact, then applying the maximum principle to the identity $$\Delta f+|\nabla f|^{2}=A_{0},$$ it is easy to see that $f$ is constant of $g$ and so
$(M,g)$ is flat. If the soliton is not flat, then we may assume that $(M,g)$ is a $\kappa$-noncollapsed, noncompact steady K\"ahler-Ricci soliton with positive Ricci curvature by the Cao's dimension reduction theorem in \cite{Ca3}.

Let $(z_1,..., z_n)$ be the Poincar\'{e} coordinates  as in Theorem \ref{Bryant} and $\phi(t)$ be a family of diffeomorphisms generated by $-2{\rm Re}(Z)=-\nabla f$. Let $p=(1,0,0,\cdots,0)$. Then one can check $z_{i}(\phi_{t}(p))=e^{-h_{i}t}z_{i}(p)$ (cf. Theorem 3 in \cite{Bry}). Namely, $Z(\phi_{t}(p))=(e^{-h_{1}t},0,\cdots,0)$.
For  $p_{k}=(k,0,\cdots,0)$   in Corollary \ref{special sequence of point},  we see  that  $p_{k}=\phi_{t_{k}}(p)$ and $t_{k}=-\frac{\ln k}{h_{1}}$. By Lemma \ref{prop-contradiction argument}, we have $R(p_{k})>C$ for some positive constant $C$ independent of $p_{k}$. On the other hand, by   (\ref {decay-scalar-t}) in  Corollary \ref{cor-of-splitting-theorem} we have
$$R(p_k)\frac{\ln k}{h_{1}}=R(p,t_{k})|t_{k}|\rightarrow1~ as ~k\rightarrow\infty.$$
Hence, we get a contradiction. The proposition is proved.
\end{proof}

Now, we prove Theorem \ref{main-theorem-nonexistence-1}.

\begin{proof} [Proof of Theorem \ref{main-theorem-nonexistence-1}]
We prove it by induction on the complex dimension of $M$. By Proposition \ref{flat-theorem}, we suppose that there is no $l$-dimensional
$\kappa$-noncollapsed steady K\"{a}hler-Ricci soliton with nonnegative sectional curvature and positive Ricci curvature for all $l<n$. To generalize the argument in the proof of Proposition \ref{flat-theorem} to higher dimensions, we only need to find a sequence of $R(p_{k})$ as in Lemma \ref{prop-contradiction argument} such that $\lim_{k\rightarrow\infty}R(p_{k})\rightarrow0$.  In fact, we  prove

\begin{claim}\label{key-lemma} Let $o$ be the unique equilibrium of $M$. Then, under the induction hypothesis, for any fixed $p\in M\setminus\{o\}$,
$R(p,-t)\to 0$ as $t\to\infty.$
\end{claim}

\begin{proof}
By Harnack inequality, we have $\frac{\partial}{\partial t}R(p,t)\ge0$.  Then $\lim_{t\rightarrow-\infty}R(p,t)$ exists   since $R(p,t)\geq0$.  Thus  there exist a point $p\in M$ such that
\begin{align}\label{eq:inf limit of R}
\lim_{t\rightarrow-\infty}R(p,t)=C>0,
\end{align}
if the claim is not true.
Consider the sequence $(M,g_{\tau}(t); p_{\tau})$, where $g_{\tau}=R(p,\tau)g(R^{-1}(p,\tau)t)$
 and $p_\tau=\phi_\tau(p).$
Then the curvature of $(M,g_{\tau}(t))$ is uniformly bounded. Note that $(M,g_{\tau}(t))$ is also $\kappa$-noncollapsed. Thus there is a subsequence $(M,g_{\tau_{i}}(t); p_{\tau_{i}})$ which  converges  to a geometric limit $(M_{\infty},g_{\infty}(t); p_{\infty})$ with  $t\in(-\infty,\infty)$.   For any fixed $t\in(-\infty,+\infty)$, it is clear  by $(\ref{eq:inf limit of R})$,
$$\lim_{\tau_{i}\rightarrow-\infty}(\tau_i+R^{-1}(p,\tau_{i})t)=-\infty.$$
Again by  $(\ref{eq:inf limit of R})$, we get
$$\lim_{\tau_{i}\rightarrow-\infty}R(p_{\tau_{i}},  R^{-1}(p,\tau_{i})t)=\lim_{\tau_{i}\rightarrow-\infty}R(p,\tau_i+R^{-1}(p,\tau_{i})t)=C.$$
Hence
\begin{align}\label{eq:constant limit}
R_{\infty}(p_{\infty},t)=\lim_{\tau_{i}\rightarrow-\infty}\frac{ R(  p_{\tau_{i}},  R^{-1}(p,\tau_{i})t)}{R(p,\tau_{i})}=1.
\end{align}
and consequently,
\begin{equation}\label{eq:property of soliton}
\frac{\partial}{\partial t}R_{\infty}(p_{\infty},t)\equiv0.
\end{equation}

By $(\ref{eq:constant limit})$, $(M_{\infty},g_{\infty}(t);p_{\infty})$ is not flat. Then by
Cao's dimension reduction theorem \cite{Ca3}, we may assume that $(M_{\infty},g_{\infty}(t))$ has positive Ricci curvature.
Since $(M_{\infty},g_{\infty}(t); p_{\infty})$ satisfies
the Harnack inequality (\ref{eq:Harnack-complex}) and there exists a point $p_{\infty}\in M_\infty$ which satisfies (\ref{eq:property of soliton}), following the argument in the proof of Theorem 4.1 in \cite{Ca1}, we
can further prove that $(M_{\infty},g_{\infty}(t);p_{\infty})$ is in fact a steady K\"{a}hler-Ricci soliton, which is $\kappa$-noncollapsed and has nonnegative sectional curvature and positive Ricci curvature.
On the other hand, by (\ref{eq:inf limit of R}), we see 
$$R(p,\tau_{i})d^{2}(o,p_{\tau_{i}})\rightarrow\infty,~as~\tau_{i}\rightarrow  \infty.$$
Then as in the proof of Theorem \ref{soliton-1}, $(M_{\infty},g_{\infty}(0))$ splits off $M_\infty=N_1\times N_2$ with
$g_\infty(0)=g_{N_1}+g_{N_{2}}$, where
$g_{N_1}={\rm d}z\otimes{\rm d}\overline{z}$ is a flat metric on $N_1$ and $g_{N_{2}}$ is a Riemannian metric on $N_2$. Consequently,
$g_{N_{2}}$ is an  $(n-1)$-dimension $\kappa$-noncollapsed  steady   K\"{a}hler-Ricci   soliton with nonnegative sectional curvature and positive  Ricci curvature.   It contradicts with the induction hypothesis. The claim is proved.

\end{proof}

As in  the proof of Proposition \ref{flat-theorem}, we  let $p=(1,0,0,\cdots,0)$.  Then $p_t=\phi_{t}(p)=(e^{-h_{1}t},0,\cdots,0)$. By Claim \ref{key-lemma}, $R(p_{t})\rightarrow0$ as $t\rightarrow -\infty$. On the other hand, $R(p_{t})=R(p,t)$ is increasing for $t\in(-\infty,+\infty)$ by the Harnack inequality. By Lemma \ref{prop-contradiction argument}, we see that
there is a positive constant $C>0$ such that $R(p_{t})\ge C$ as long as $-t$ sufficiently large. Therefore, we get a contradiction. The proof of Theorem \ref{main-theorem-nonexistence-1} is complete.
\end{proof}

By Theorem \ref{main-theorem-nonexistence-1} together with Cao's dimension reduction theorem \cite{Ca3}, we get immediately

\begin{cor}\label{rigidity-soliton}Any $n$-dimensional $\kappa$-noncollapsed steady K\"ahler-Ricci soliton with non-negative sectional curvature must be flat.
\end{cor}

At the end, we apply Corollary \ref{rigidity-soliton} to prove Theorem \ref{theorem-eternal-flow}.

\begin{proof}[Proof of Theorem \ref{theorem-eternal-flow}]
We only need to prove that $R(p,t)\equiv0$ for all $p\in M$ and $t\in(-\infty,+\infty)$. Suppose not. Fix any $p\in M$ such that $R(p,t^{\prime})>0$ for some $t^{\prime}\in(-\infty,+\infty)$. Let $\{t_{k}\}$ be a sequence of numbers which tends to infinity and $g_{k}(t)=g(t+t_{k})$. Since each flow $(M,g_{k}(t); p)$ is $\kappa$-noncollapsed and has uniformly bounded curvature, $(M,g_{k}(t);  p)$ converges to $(M_{\infty},g_{\infty}(t); p_{\infty})$ in the Cheeger-Gromov topology.
Note that the Harnack inequality (\ref{eq:Harnack-complex}) holds along flow $(M,g(t))$ and $(M,g(t))$ has uniformly bounded curvature. Thus $\frac{\partial}{\partial t}R(p,t)\ge0$ and $R(p,t)$ is uniformly bounded. It follows
$$R_{\infty}(p_{\infty},t_{1})=\lim_{t\rightarrow\infty}R(p,t+t_{1})=\lim_{t\rightarrow\infty}R(p,t+t_{2})=R_{\infty}(p_{\infty},t_{2}),$$
and consequently,
\begin{equation}\label{maximal-point}
\frac{\partial}{\partial t}R_{\infty}(p_{\infty},t)\equiv0.
\end{equation}

Since $R_{\infty}(p_{\infty},t)\ge R(p,t^{\prime})>0$, $(M_{\infty},g_{\infty}(t))$ is non-flat. By Cao's dimension reduction theorem in \cite{Ca3}, we may assume that $(M_{\infty},g_{\infty}(t))$ has positive Ricci curvature. Since $(M_{\infty},g_{\infty}(t);p_{\infty})$ satisfies
the Harnack inequality (\ref{eq:Harnack-complex}) and there exists a point $p_{\infty}\in M_\infty$ which satisfies (\ref{maximal-point}), following the argument in the proof of Theorem 4.1 in \cite{Ca1}, we
can prove that $(M_{\infty},g_{\infty}(t); p_{\infty})$ is in fact a (gradient) steady K\"{a}hler-Ricci soliton, which is $\kappa$-noncollapsed and has nonnegative sectional curvature. By Corollary \ref{rigidity-soliton}, $(M_{\infty},g_{\infty}(t);  p_{\infty})$ is a flat metric flow. This is impossible because $R_{\infty}(p_{\infty},t)\ge R(p,t^{\prime})>0$. Hence, we complete the proof.
\end{proof}

\section{Appendix}
In this appendix, we compute the curvature decay of the steady gradient K\"{a}her-Ricci solion on $\mathbb C^n$ constructed by Cao in \cite{Ca2}, and show that these steady soltions are collapsed.

We first recall Cao's construction. Let $(z_{1},z_{2},\cdots,z_{n})$ be the standard holomorphic coordinates on $\mathbb{C}^{n}$. Assume that $g=(g_{i\bar{j}})$ is an U(n)-invariant metric on $\mathbb{C}^{n}$ and the corresponding K\"{a}hler potential is given by $u(s)$, where $u(s)$ is a strictly increasing and convex function on $(-\infty,\infty)$ and $s=\ln|z|^{2}=\ln r^2$. By a direct computation, we have
\begin{align}
g_{i\bar{j}}=\partial_{i}\partial_{\bar{j}}u(s)=e^{-s}u'(s)\delta_{ij}+e^{-2s}\bar{z}_{i}z_{j}(u''(s)-u'(s)),\notag
\end{align}
\begin{align}
g^{i\bar{j}}=\partial_{i}\partial_{\bar{j}}u(s)=e^{s}u'(s)^{-1}\delta_{ij}+z_{i}\bar{z}_{j}(u''(s)-u'(s)),\notag
\end{align}
and
\begin{align}\label{soliton-cn-det}f(s)\triangleq -\ln\det(g_{i\bar{j}})=ns-(n-1)\ln u'(s)-\ln u''(s).
\end{align}
Then
\begin{align}\label{ricci-rotation}
R_{i\bar{j}}=\partial_{i}\partial_{\bar{j}}f(s)=e^{-s}f'(s)\delta_{ij}+e^{-2s}\bar{z}_{i}z_{j}(f''(s)-f'(s)).
\end{align}
Thus $g_{i\bar{j}}$ is a steady gradient soliton if and only if
\begin{equation*}
v^{i}\frac{\partial}{\partial z_i}= g^{i\bar{j}}\partial_{\bar{j}}f\frac{\partial}{\partial z_i}=(z_{i}\frac{f'}{u''})\frac{\partial}{\partial z_i}\notag
\end{equation*}
is a holomorphic vector field, which is equivalent to
\begin{equation}\label{soliton-cn-1}
f'=\lambda u'',
\end{equation}
for some constant $\lambda$.

Let $\phi=u'$.
Then by $(\ref{soliton-cn-1})$ and $(\ref{soliton-cn-det})$, we get an equation for $\phi$,
\begin{equation}\label{equation-phi}
\phi^{n-1}\phi^{\prime}e^{\alpha\phi}=\beta e^{ns}.
\end{equation}
After rescaling, we may choose $\alpha=\beta=1$.
Cao solved (\ref{equation-phi}) by
\begin{equation}\label{soliton-cn-3}
\sum_{k=0}^{n-1}(-1)^{n-k-1}\frac{n!}{k!}\phi^{k}e^{\phi}=e^{ns}+(-1)^{n-1}n!.
\end{equation}
Cao has observed the following properties of $\phi$,
\begin{align}\label{r-behavior}
&\phi(s)>0,~\phi^{\prime}(s)>0,~\forall~s\in(-\infty,+\infty), \notag\\
&\lim_{s\rightarrow\infty}\frac{\phi(s)}{s}=n,~\lim_{s\rightarrow\infty}\phi^{\prime}(s)=n.
\end{align}
He also proved that these solitons has positive sectional curvature.

The curvature asymptotic behavior can be also computed in the following. Let $o=(0,0,\cdots,0)$ and $p=(z_{1},0,\cdots,0)$. Then
by (\ref{ricci-rotation}), we have
\begin{align} R(p)&=-\frac{1}{\phi'}\Big((n-1)(\frac{\phi'}{\phi})'+(\frac{\phi''}{\phi'})'\Big)+\frac{n-1}{\phi}\Big(n-(n-1)\frac{\phi'}{\phi}-\frac{\phi''}{\phi'}\Big)\label{soliton-cn-4}\\
&=n-\phi^{\prime}. \notag\end{align}
On the other hand, by differentiating (\ref{soliton-cn-3}), it follows
\begin{equation}
\phi^{\prime}=\frac{e^{ns}}{e^{ns}+(-1)^{n-1}n!}\sum_{k=0}^{n-1}\Big((-1)^{n-k-1}\frac{n!}{k!}\phi^{k-n+1}\Big).\notag
\end{equation}
Thus
\begin{align}\label{scalar-cao}
R(p)=&\frac{(-1)^{n-1}n!\cdot n}{e^{ns}+(-1)^{n-1}n!}\notag\\
&+\frac{e^{ns}}{e^{ns}+(-1)^{n-1}n!}\Big(\frac{n(n-1)}{\phi}+\frac{1}{\phi^{2}}\sum_{k=0}^{n-3}(-1)^{n-k-1}\frac{n!}{k!}\phi^{k-n+3}\Big)\notag\\
&\to (n-1), ~{\rm as}~|z_1|\to\infty.
\end{align}

Let $\rho(x)$ be a distance function from the original point $o\in \mathbb C^n$. Then by (\ref{r-behavior}), it is easy to see
\begin{align}\label{distance-r}
\rho(x)=\frac{\sqrt{n}}{2}s(x)+o(s(x)),~as~s\rightarrow\infty.
\end{align}
Hence, using the U(n)-symmetry of $g$, we obtain from (\ref{scalar-cao}),

\begin{lem}\label{lamma-append} The metric $g$ satisfies the following curvature condition,
\begin{align}\label{pointwise-decay}
R(x)\rho(x)\rightarrow \frac{1}{2}\sqrt{n}(n-1),~as~|x|\rightarrow\infty.
\end{align}

\end{lem}

By Lemma \ref{lamma-append}, we prove

\begin{prop}Any $U(n)$-symmetric steady gradient soliton on $\mathbb C^n$
is collapsed.
\end{prop}

\begin{proof}
Let $z_{i}=x_{i}+\sqrt{-1}y_{i}$ for $1\le i\le n$. We introduce new coordinates $(r,\theta,x_{2}^{\prime},y_{2}^{\prime},\cdots,x_{n}^{\prime},y_{n}^{\prime})$ such that
\begin{equation*}
\left\{ \begin{aligned}
x_{1} &= \cos\theta\sqrt{r^{2}-\Sigma_{i=2}^{n}(x_{i}^{2}+y_{i}^{2})}, \\
y_{1} &= \sin\theta\sqrt{r^{2}-\Sigma_{i=2}^{n}(x_{i}^{2}+y_{i}^{2})}, \\
x_{2} &= rx_{2}^{\prime},\\
y_{2} &= ry_{2}^{\prime},\\
\cdots\\
x_{n} &= rx_{n}^{\prime},\\
y_{n} &= ry_{n}^{\prime}.
\end{aligned} \right.
\end{equation*}
Then under the new coordinates the metric $g$ has an expression,
\begin{align}
g&=r^{-2}\phi^{\prime}(s)({\rm d}r^{2}+r^{2}{\rm d\theta^{2}})+\phi(s)\pi^{*}g_{FS}\notag\\
&=\frac{\phi^{\prime}(s)}{4}{\rm d}s^{2}+\phi^{\prime}(s){\rm d\theta^{2}}+\phi(s)\pi^{*}g_{FS}\notag\\
&=\phi^{\prime}(\tau^{2})\tau^{2}{\rm d}\tau^{2}+\phi^{\prime}(\tau^{2}){\rm d\theta^{2}}+\phi(\tau^{2})\pi^{*}g_{FS},\label{new-version-g-1}
\end{align}
where $\pi: S^{2n-1}\to \mathbb{CP}^{n-1}$ is the $S^{1}$-Hopf fibration.
Let $p_{k}\in M$ such that $|p_{k}|^{2}=e^{k^{2}}$ and let $r_{k}= \frac{k}{2\sqrt{n-1}}$. By the choice of $p_{k}$, we have
$s(p_{k})=k^{2}.$

Let $N_{k}=\{x\in M: k^{2}-k\le s(x)\le k^{2}+k\}$ and $g_{k}=\phi(p_{k})^{-1}g$. We consider open manifolds $(N_{k}, g_{k})$. By the asymptotic behavior of $\phi(s)$ and (\ref{new-version-g-1}),
it is easy to see that $(N_{k}, g_{k})$ converge to $(\mathbb{R}\times \mathbb{CP}^{n-1}, ds^2\otimes g_{FS} )$ in $C^\infty$ topology. Note that $B(p_{k},r_{k})\subset N_{k}$.  By the convergence, for any $x\in B(p_{k},r_{k})$, $s(x)\in [k^{2}-2r_{k},k^{2}+2r_{k}]$ and $(x_{2}^{\prime}(x),y_{2}^{\prime}(x),\cdots,x_{n}^{\prime}(x),y_{n}^{\prime}(x))\subset B_{FS}(p_{k},2\phi(p_{k})^{-1/2}r_{k})$, where $B_{FS}(p_{k},r)$ is the geodesic ball of the submanifold $\{(r(p_{k}),\theta(p_{k}),x_{2}^{\prime},y_{2}^{\prime},\cdots,x_{n}^{\prime},y_{n}^{\prime})\in M\}$ with the metric $\pi^*g_{FS}$. Hence, the volume of $B(p_{k},r_{k})$ satisfies the following estimate for  sufficiently large  $k$,
\begin{align}\label{local-volume}
& {\rm vol}(B(p_{k},r_{k}))\notag\\
&\le\int_{k^{2}-2r_{k}}^{k^{2}+2r_{k}} ds\int_{0}^{2\pi} d\theta \int_{B_{FS}( p_k,  2\phi(p_{k})^{-1/2}r_{k})}\phi'(s)\phi(s)^{n-1} d{\rm vol}_{g_{FS}}\notag\\
&=2\pi  (\phi(p_{k}))^{n-1}    \int_{k^{2}-2r_{k}}^{k^{2}+2r_{k}} ds      \int_{B_{FS}( p_k,  2\phi(p_{k})^{-1/2}r_{k})}   \phi'(s)\Big(\frac{\phi(s)}{\phi(p_{k})}\Big)^{n-1} d{\rm vol}_{g_{FS}} \notag\\
&\leq 2\pi   (\phi(p_{k}))^{n-1}    \int_{k^{2}-2r_{k}}^{k^{2}+2r_{k}} ds   \int_{ \mathbb{CP}^{n-1}}   2^{n-1}n d{\rm vol}_{g_{FS}}\notag\\
&\leq (32)^{n+1}n (n-1)^{n-1}\pi  \omega_{2n-2} r^{2n-1}_{k}.
\end{align}
It follows
$$\lim_{k\rightarrow\infty}\frac{{\rm vol}(B(p_{k},r_{k}))}{r_{k}^{2n}}=0.$$
On the  other hand, by Lemma \ref{lamma-append},
$$R(x)\le  \frac{2(n-1)}{k^2-k}\le \frac{4(n-1)}{k^2}=\frac{1}{r_{k}^{2}},~ \forall~ x\in B(p_{k},r_{k}),$$
 when $k$ is large enough.
Hence $g$ is collapsed.
\end{proof}

From the computation in (\ref{local-volume}), it is easy to get the volume growth of $B(p, r)$,
$${\rm vol}(B(p, r))=O(r^n),~{\rm as}~r\to\infty.$$

\end{document}